\documentclass[11pt]{amsart}

\title{The complexity of classification problems for models of arithmetic}
\author{Samuel Coskey}
\author{Roman Kossak}

\usepackage[marginratio=1:1]{geometry}
\usepackage{setspace}
\usepackage{mathpazo,bm}
\usepackage{romandefs}

\def\LO{{\sf LO}}
\def\sset{{\sf set}}
\def\srec{{\sf rec}}
\def\sfg{{\sf fg}}
\def\iso{\cong}
\def\oiso{\mathord{\cong}}
\def\olt{\mathord{<}}
\def\oeq{\mathord{=}}
\def\defeq{\mathrel{\mathpalette{\vcenter{\hbox{$:$}}}=}}

\begin{document}
\begin{abstract}
  We observe that the classification problem for countable models of
  arithmetic is Borel complete.  On the other hand, the classification
  problems for finitely generated models of arithmetic and for
  recursively saturated models of arithmetic are Borel; we investigate
  the precise complexity of each of these.  Finally, we show that the
  classification problem for pairs of recursively saturated models and
  for automorphisms of a fixed recursively saturated model are Borel
  complete.
\end{abstract}
\maketitle

\section{Introduction}

It is well-known that models of Peano Arithmetic (\PA) are highly
unclassifiable. In this note, we aim to make this statement more
precise by showing that many natural classification problems related
to countable nonstandard models are of high complexity according to
the descriptive set theory of equivalence relations. Our main tool
will be Gaifman's minimal types of \cite{gai76}, which provide a
method of constructing models of \PA\ ``along'' linear orders. The
book \cite{ksbook} provides all of the necessary details of this
method, as well as the background concerning \rs\ models. The
model-theoretic arguments that we shall use are standard. We will try
to give enough details in our arguments so that readers unfamiliar
with models of arithmetic can understand the most important special
cases.

In order to rigorously discuss the complexity of classification
problems, we must use the language of \emph{Borel equivalence
  relations}, an area of descriptive set theory.  This subject was
initiated in \cite{friedmanstanley} and \cite{hjorthkechris}, and a
good introduction can be found in \cite{kanovei}. To explain how it
applies, we will demonstrate how each of the classification problems
which we shall consider (along with a great many others) can be
identified with an equivalence relation on some standard Borel
space. Recall that a \emph{standard Borel space} is complete separable
metric space equipped just with its $\sigma$-algebra of Borel sets.
The most important example for us is the following. If $\lan$ is a
countable relational language and $\Theta$ is an $\lan$-theory (or
more generally, a sentence of the infinitary language
$\lan_{\om_1,\om}$ in which infinite conjunctions and disjunctions are
allowed), then the set
\[X_\Theta\defeq
  \set{M:\textrm{the domain of $M$ is $\omega$ and $M\models\Theta$}}
\]
is called the \emph{space of countable models of
  $\Theta$}.\footnote{More precisely, if $a(R)$ denotes the arity of
  $R\in\lan$, then $X_\Theta$ can be regarded as a Borel subset of the
  space $\prod_{R\in\lan}\mathcal{P}(\omega^{a(R)})$ of all
  $\lan$-structures with domain $\omega$.  It follows from the general
  theory that $X_\Theta$ is a standard Borel space in its own right.}
Studying the classification problem for countable $\Theta$-models now
amounts to studying the \emph{isomorphism equivalence relation}
$\oiso_\Theta$ on $X_\Theta$.

Now, if $E,F$ are (not necessarily Borel) equivalence relations on the
standard Borel spaces $X,Y$, then we say that $E$ is \emph{Borel
  reducible} to $F$ (written $E\leq_BF$) iff there exists a Borel
function $f:X\rightarrow Y$ such that
\[x\mathrel{E}x'\Iff f(x)\mathrel{F}f(x')\;.
\]
The function $f$ is said to be a \emph{Borel reduction} from $E$ to
$F$.  Informally, we take $E\leq_BF$ to imply that the classification
problem for elements of $Y$ up to $F$ is at least as hard as the
classification problem for elements of $X$ up to $E$.

\begin{defn}
  Let $\lan$ be a countable language and $\Theta$ a sentence of
  $\lan_{\om_1,\om}$.  The class of $\Theta$-models is said to be
  \emph{Borel complete} iff for any $\lan'$ and any sentence $\Theta'$
  of $\lan'_{\omega_1,\omega}$, we have
  $\oiso_{\Theta'}\leq_B\oiso_{\Theta}$.
\end{defn}

We remark that the terminology is unfortunately misleading, since if
$\oiso_\Theta$ is the isomorphism relation for a Borel complete class,
then $\oiso_\Theta$ is a properly analytic set pairs.  The Borel
complete equivalence relations form a single bireducibility class
which is of course quite high in the $\leq_B$ hierarchy.  Many
familiar classes are known to be Borel complete. For some examples, it
is shown in \cite{friedmanstanley} that the class of countable groups,
of countable connected graphs, and of countable linear orders are all
Borel complete.

In the next section, we shall show that the classification problem for
countable models of arithmetic is also Borel complete.  Afterwards, we
turn our attention to the classification problems for various
important collections of countable models of \PA.  In the third
section, we consider the class of finitely generated models, and in
the fourth the recursively saturated models.  The classification
problem for each of these classes of models is Borel.

In the final two sections, we consider the isomorphism problem for
particular expansions of models of \PA.  In the fifth section, we
shall show that the classification problem for elementary pairs of
recursively saturated models is Borel complete.  As an application, we
show in the last section that the conjugacy problem for automorphisms
of a recursively saturated model is also Borel complete.

We would like to thank Jim Schmerl for his careful reading of the
preliminary version of this paper. Jim caught a serious error and
replaced it with an interesting result (contradicting our erroneous
claim) which is Theorem \ref{fg} below. The theorem and its proof are
presented here with his kind permission.

\section{Canonical $I$-models}

In this section, we will outline how Gaifman used minimal types to
build canonical models of \PA\ along a given linear order (refer to
Section~3.3 of \cite{ksbook} for the full details).  From the details
of this construction, we shall see that the isomorphism relation for
countable linear orders is Borel reducible to the that for countable
models of \PA\, and hence that the class of countable models of \PA\
is Borel complete.  Lastly, we will give some additional facts
concerning these canonical models that will be useful in later
sections.

Minimal types were originally defined by Gaifmain, and they are
so-named because he used them to obtain minimal elementary extensions.
We omit the original definition, but use instead the characterization
that $p(x)$ is \emph{minimal} iff it satisfies the following two
properties:
\begin{enumerate}
\item \emph{unbounded}: $(t<x)\in p(x)$ for each closed Skolem term
  $t$, and
\item \emph{indiscernible}: for every model $M$, and all sequences
  $a_1<\cdots<a_n$ and $b_1<\ldots<b_n$ of realizations $p(x)$ in $M$, we
  have $(M,\bar a)\equiv (M,\bar b)$.
\end{enumerate}
It is not difficult to construct minimal types by repeatedly applying
Ramsey's theorem, formalized in \PA.

We now show how to to build, given a linear order $I$ and a completion
$T$ of \PA, the \emph{canonical $I$-model} of $T$, which we shall
denote $M_T(I)$.  First, fix a minimal type $p(x)$.  There are $\cont$
many such types and it is not important which one we pick, but to make
our constructions parameter-free, we can always choose one which is
uniformly arithmetic in $T$. Next, form the type
\[\Delta(x_i)_{i\in I}\defeq\left(\bigcup_{i\in I}p(x_i)\right)
\cup\set{x_i<x_j: i,j\in I\land i<j}\;,
\]
and let $M_T(I)$ be the Skolem closure of a sequence realizing
$\Delta(x_i)_{i\in I}$.  Now, the key point is that in fact, the
ordertype of $I$ is determined by $M_T(I)$.  (More specifically, by
Theorem~3.3.5 of \cite{ksbook}, the ordertype of $I$ can be recovered
from the ordertype of the set of \emph{gaps} in $M_T(I)$.)  In
particular, we have that if $(I,<)$ and $(J,<)$ are linear orders then
\[(I,<)\cong (J,<)\Iff M_T(I)\cong M_T(J)\;.
\]

Although this works for linear orders of any cardinality, it is easy
to see that for countable $I$, the construction of $M_T(I)$ is
arithmetic in $T$ and $I$.  In particular, there exists a Borel
function $f$ from the space of countable linear orders to the space of
countable $T$-models such that whenever $x=(\omega,\mathord{<})$ is of
ordertype $I$ then $f(x)$ is of isomorphism type $M_T(I)$.

We have established the following result.  Let $\oiso_\LO$ denote the
isomorphism equivalence relation on the space of countable linear
orders and $\oiso_T$ the isomorphism equivalence relation on the space
of countable $T$-models.

\begin{thm}
  \label{first}
  There exists a Borel reduction from $\oiso_\LO$ to $\oiso_T$ which
  sends a linear order $I$ to a canonical $I$-model of $T$. In
  particular, $\oiso_T$ is Borel complete.
\end{thm}

We remark that for $I$ infinite, the model $M_T(I)$ is not finitely
generated. It follows from the results of the next section that the
use of non-finitely generated models is essential for
Theorem~\ref{first}.

In later sections, we shall require another important feature of
canonical $I$-models.  First, recall that an extension $K\prec M$ is
said to be an \emph{end extension}, written $K\ee M$, iff $K$ is an
initial segment of $M$.  Next, for a structure $M$, we let $\Def(M)$
denote collection of all subsets of $M$ which are definable from
parameters in $M$. An extension $K\prec M$ is said to be
\emph{conservative} iff for every $X\in\Def(M)$, we have that $X\cap
K\in\Def(K)$.  The MacDowell-Specker Theorem (see for instance
Theorem~2.2.8 of \cite{ksbook}) states that any model of \PA\ has a
conservative elementary end extension.  The following classical
result, due to Gaifman, can again be found in more detail in
Section~3.3 of \cite{ksbook}.

\begin{thm}
  \label{cons}
  Let $M_T(I)$ be the canonical $I$ model and suppose that $J$ is a
  proper initial segment of $I$. Then $M_T(I)$ is a conservative
  elementary end extension of $M_T(J)$.
\end{thm}

For further applications, let us note that the construction of
$M_T(I)$ works in a much more general context.  For
$\lan\supset\{\mathord{+},\mathord{\times},0,1\}$ a \ct\ language, we
let $\PA(\lan)$ be the theory obtained from \PA\ by adding instances
of the induction schema for all $\lan$-formulas.  We will use the
notation $\SPA$ as a stand-in for $\PA(\lan)$ for any countable
$\lan$.  The construction of canonical $I$-models can also be carried
out for models of $\SPA$, and everything which has been said in this
section holds in this general situation as well.

\section{Finitely generated models}

While each $\oiso_T$ is Borel complete, the isomorphism equivalence
relation $\oiso_T^\sfg$ on the space of \emph{finitely generated}
models of $T$ is Borel.  In this section, we shall see that according
to the $\leq_B$ hierarchy, $\oiso_T^\sfg$ lies among the countable
Borel equivalence relations.  After introducing this important class,
we present a theorem of Schmerl which helps us further understand the
complexity of $\oiso_T^\sfg$.

For each \ar\ formula $\vp(x,\bar y)$ there is a corresponding Skolem
term $t_\vp(\bar y)$, which is defined to be $\min\{x:\vp(x,\bar y)\}$
if this set is nonempty, and 0 otherwise. If $M$ is a model of \PA,
then the \emph{Skolem closure} of some $\bar a\in M^n$ is the set
\[\hull(\bar a)=\{t(\bar a): t \textup{ is a \sk}\}\;.
\] $M$ is said to be \emph{finitely generated} iff there is an $\bar
a\in M^n$ such that $M=\hull(\bar a)$. Since it is possible to code a
finite sequence of natural numbers as a single natural number, and
this can be done definably in \PA, we can always suppose that a
finitely generated model is generated by a single element. If $M$ and
$N$ are finitely generated, then $M\cong N$ iff there are $a$ and $b$
such that $M=\hull(a)$, $N=\hull(b)$ and $\tp(a)=\tp(b)$.  It is easy
to see that the condition on the right hand of this equivalence is
Borel (in $M$ and $N$).

The observation that $\oiso_T^\sfg$ is Borel, combined with
Theorem~\ref{first}, already yields an interesting corollary. By
Theorem~2.1.12 of \cite{ksbook}, every \ct\ model $M$ of $T$ has a
finitely generated minimal \ele.  The construction used in the proof
Theorem~2.1.12 of \cite{ksbook} is not canonical, it depends on the
choice of enumeration of the model $M$. The following result shows
that in fact there is no canonical construction.

\begin{cor}
  Let $T$ be a completion of \PA.  Then there is no Borel map taking
  each \ct\ model $M$ of $T$ to a finitely generated minimal \ele\ of
  $M$.
\end{cor}

\begin{proof}
  Suppose that $f$ is such a map. If both extensions $M\ee M'$ and
  $N\ee N'$ are minimal, then any isomorphism between $M'$ and $N'$
  must map $M$ onto $N$. Hence if $M$ and $N$ are nonisomorphic, then
  $M'$ and $N'$ are nonisomorphic. It follows that $f$ is in fact a
  Borel \emph{reduction} from $\oiso_T$ to $\oiso_T^\sfg$. Hence, the
  composition of $f$ with the Borel reduction $(I,\mathord{<})\mapsto
  M_T(I)$ given by Theorem~\ref{first} would yield a Borel reduction
  from $\oiso_\LO$ to $\oiso_T^\sfg$.  But this is impossible, since
  $\oiso_T^\sfg$ is Borel complete, and a Borel complete equivalence
  relation cannot be Borel.
\end{proof}

We next observe that $\oiso_T^\sfg$ has the stronger property that it
is essentially countable.  Here, a Borel equivalence relation $E$ is
called \emph{countable} iff every $E$-class is countable, and $E$ is
called \emph{essentially countable} iff it is Borel bireducible with a
countable Borel equivalence relation.  Let us also say that a class
$\mathcal C$ of countable models is essentially countable iff the
isomorphism equivalence relation $\oiso_C$ on $\mathcal C$ is
essentially countable.  We will need the following characterization
from \cite{hjorthkechris} of the essentially countable classes.

\begin{thm}[Hjorth-Kechris]
  \label{hk}
  Let $\Theta$ be a sentence of $\lan_{\om_1,\om}$.  Then the class of
  models of $\Theta$ is essentially countable iff there is a \ct\
  fragment $F$ of $\lan_{\om_1,\om}$ with $\Theta\in F$ such that for
  every countable $M\models\Theta$ there exists $n\in\omega$ and $\bar
  a\in M^n$ such that $\Th_F(M,\bar a)$ is $\aleph_0$-categorical.
\end{thm}

Many classes of models which are finitely generated in some sense turn
out to be essentially countable.  For instance, the class of finitely
generated groups is essentially countable, as is the class of fields
of finite transcendence degree.

\begin{prop}
  \label{fg_ec}
  $\oiso_T^\sfg$ is essentially countable.
\end{prop}

\begin{proof}
  Let $\Theta$ be the conjunction of the axioms of \PA\ together with
  the sentence
  \[\ex x\fo y\bigvee\{y=t(x):t \textup{ is a \sk}\}\;.
  \]
  If $F$ is any countable fragment of $\lan_{\om_1,\om}$ containing
  $\Theta$, then the sentence
  \[\fo y\bigvee\{y=t(a):t \textup{ is a \sk}\}
  \]
  is in $\Th_F(M, a)$, and the result follows from Theorem~\ref{hk}.
\end{proof}

We now briefly discuss the structure of the countable Borel
equivalence relations.  Here, we will work only on \emph{uncountable}
standard Borel spaces; it is a classical result that there is a unique
such space up to Borel bijections.  By a theorem of Silver, the
\emph{equality} equivalence relation $\oeq_{2^\omega}$ on $2^\omega$
is the least complex countable Borel equivalence relation.  An
equivalence relation $E$ which is Borel reducible to $\oeq_{2^\omega}$
is called \emph{smooth}, or \emph{completely classifiable} because the
Borel reduction gives a system of complete invariants for the
classification problem up to $E$.  The next least complex equivalence
relation is the \emph{almost equality} relation $E_0$ on $2^\omega$
defined by $x\mathrel{E}_0x'$ iff $x(n)=x'(n)$ for all but finitely
many $n$.  By Harrington-Kechris-Louveau \cite{hkl}, a Borel
equivalence relation $E$ is nonsmooth iff $E_0\leq_B E$.

It also turns out that there exists a \emph{universal} countable Borel
equivalence relation, which we denote by $E_\infty$.  For instance,
the class of finitely generated groups lies at the level of
$E_\infty$, as does the class of connected locally finite graphs.  It
seems likely that $\oiso_T^\sfg$ is also bireducible with $E_\infty$,
but we don't know this yet for sure.  We now present an argument of
Schmerl which at least eliminates the possibility that $\oiso_T^\sfg$
is smooth.

\begin{thm}
  \label{fg}
  If $T$ is any completion of \PA, then $E_0$ is Borel reducible to
  $\oiso_T^\sfg$.
\end{thm}

For the proof, let $M$ be a prime model of $T$ and let $G$ be the
group of definable permutations of $M$.  Then $G$ acts on the space
$S(T)$ of complete $1$-types over $T$ by setting $gp(x)=$ the unique
complete type in $S(T)$ containing $\set{\vp(g^{-1}(x)):\vp(x)\in
  p(x)}$.  (Here, each $g\in G$ is identified with a Skolem term for
$g$.)  Notice that if $p(x)$ is the type of $a$, then $gp(x)$ is the
type of $g(a)$.  Let $E^{S(T)}_G$ denote the orbit equivalence
relation on $S(T)$ induced by the action of $G$.

\begin{lem}
  $\oiso_T^\sfg$ is Borel bireducible with $E^{S(T)}_G$.
\end{lem}

Thus, we have found an explicit countable relation witnessing that
$\oiso_T^\sfg$ is essentially countable.


\begin{proof}
  It is not difficult to show that any map which sends a type $p(x)$
  to a canonically defined prime model of $p(x)$ will give a reduction
  from $E^{S(T)}_G$ to $\oiso_T^\sfg$.  Similarly, any map which sends
  a finitely generated model of $T$ to the type of one of its
  generators will give a reduction from $\oiso_T^\sfg$ to $E^{S(T)}_G$.
\end{proof}

Theorem~\ref{fg} now follows immediately from the following result.

\begin{lem}
  $E_0$ is Borel reducible to $E^{S(T)}_G$.
\end{lem}

\begin{proof}
  We will construct a family $\seq{X_s:s\in2^{\olt\omega}}$ of
  unbounded definable subsets of $M$ with the following properties:
  \begin{enumerate}
  \item $X_s\subset X_t$ whenever $s\supset t$;
  \item $X_s\cap X_t=\varnothing$ whenever $\abs{s}=\abs{t}$ and $s\neq
    t$;
  \item for every $b\in 2^\omega$, there exists a unique type $p_b(x)$
    such that $X_{b\harpo n}$ is in $p_b(x)$ for all $n\in\om$.
    (Here, we say that a definable set $X$ is in $p$ iff the formula
    that defines it is in $p$.);
  \item for every $b,b'\in2^\omega$, we have $b\mathrel{E}_0b'$ iff
    $p_b\sim_Tp_{b'}$.
  \end{enumerate}
  Thanks to property (4), the proof will be complete once this is
  done.  Our construction will have the following additional property.
  First, for all $s,t\in2^{\mathord{<}\omega}$ such that
  $\abs{s}=\abs{t}$, let $\alpha_{s,t}\colon X_s\rightarrow X_t$
  denote the unique definable order-preserving bijection.  Then we
  will have:
  \begin{enumerate}
  \item[(5)] $\alpha_{s,t}\harpo X_{sr}=\alpha_{sr,tr}$ for all $s,t$
    such that $\abs{s}=\abs{t}$, and for all $r$.
  \end{enumerate}

  To begin the construction, let $\seq{\phi_i(x,y): i\in\om}$ be a
  fixed enumeration of the binary formulas.  Let $X_\varnothing=M$,
  and given $X_s$ for all $s\in2^n$, we define $X_{s0},X_{s1}$ as
  follows.  First, repeatedly using Ramsey's Theorem (formalized
  inside \PA) and the functions $\alpha_{s,t}$, we find unbounded
  definable subsets $Y_s\subset X_s$ such that
  \begin{itemize}
  \item[(6)] $Y_s$ is homogeneous for $\phi_n(x,\al_{s,t}(y))$ for all
    $s,t\in2^n$; (Here, $Y$ is said to be {\em homogeneous} for
    $\vp(x,y)$ iff for all $x,y,u,v\in Y$ with $x<y$ and $u<v$, we
    have $(\vp(x,y)\iff
    \vp(u,v))\;\wedge\;(\vp(y,x)\iff\vp(v,u))\;\wedge\;(\vp(x,x)\iff\vp(u,u))$.)
  \item[(7)] $\alpha_{s,t}(Y_s)=Y_t$ for all $s,t\in2^n$.
  \end{itemize}
  Next, let $X_{s0},X_{s1}$ be a partition of $Y_s$ into disjoint
  unbounded and definable sets (you could take every other element in
  an enumeration in $Y_s$).  Now, for each $b\in 2^\om$ we define
  $p_b(x)$ by
  \[\vp(x)\in p_b(x)\Iff\ex n\ X_{b\harpo n}\subseteq \vp(M)\;.
  \]

  Thus we can guarantee that (1)--(3) and (5)--(7) are all satisfied;
  it remains only to show that (4) follows from these.  Suppose first
  that $b\mathrel{E}_0b'$, and let $n\in\omega$ be the last index such
  that $b(n-1)\neq b'(n-1)$.  Then by (5), $\alpha_{b\harpo n,b'\harpo
    n}$ maps $X_{b\harpo i}$ onto $X_{b'\harpo i}$ for all $i\geq n$.
  It is not difficult to extend $\alpha_{b\harpo n,b'\harpo n}$ to a
  definable permutation of $M$ which also maps $X_{b\harpo i}$ onto
  $X_{b'\harpo i}$ for all $i<n$.  It follows that
  $p_b(x)\sim_Tp_{b'}(x)$.

  For the converse, suppose that $p_b(x)\sim_Tp_{b'}(x)$ and let $g\in
  G$ be a definable permutation of $M$ satisfying $gp_b(x)=p_{b'}(x)$.
  Then we have:
  \begin{center}
    for all definable $X$, if $X\in p_b(x)$ then $g(X)\in p_{b'}(x)$.
  \end{center}
  Let $n$ be such that $\vp_n(x,y)$ is the formula for $g(x)=y$, and
  let $s=b\harpo n$ and $t=b'\harpo n$. Then $Y_s$ is homogeneous for
  $\vp_n(x,\al_{s,t}(y))$. Since $\vp_n(x,\al_{s,t}(y))$ defines the
  relation $g(x)=\al_{s,t}(y)$, by (6) one of the following holds:
  \begin{enumerate}
  \item[(a)] For all $x\in Y_s$, $g(x)=\al_{s,t}(x)$, or
  \item[(b)] For all $x,y\in Y_s$, if $x\not=y$, then
    $g(x)\not=\al_{s,t}(y)$.
  \end{enumerate}
  But (b) implies that $g$ sends $Y_s$ completely outside of $Y_t$,
  contradicting that $p_{b'}(x)=gp_b(x)$.  Thus (a) holds, and this
  implies that $g\harpo Y_s=\al_{s,t}\harpo Y_s$. It follows that
  $\al_{s,t}$ maps $X_{b\harpo i}$ to $X_{b'\harpo i}$ for all $i>n$,
  and together with (5) this implies that $b(i)=b'(i)$ for all $i>n$.
  Thus, $b\mathrel{E}_0b'$, and the proof is complete.
\end{proof}

It is worth remarking that as a consequence of property (6) of the
above construction, the types $p_b$ are each unbounded and
2-indiscernible.  Such types are indiscernible and minimal in the
sense of Gaifman. Since minimal types are extremely special, this
gives some evidence that $E^{S(T)}_G$ is much more complex than $E_0$.

\section{Recursively saturated models}

Let $\lan$ be a finite first-order language. An $\lan$-structure is
\emph{\rs} iff for any finite $\bar a\in M^n$, and any recursive set
of $\lan$-formulas $p(x,\bar y)$, if $p(v,\bar a)$ is consistent with
$\Th(M,\bar a)$, then $p(v,\bar a)$ is realizable in $M$.  Countable
\rs\ models of \PA\ form a robust class which has been intensively
studied over the last 30 years. In this section we shall show that, in
contrast with the class of all \ct\ models of \PA, the classification
problem for the \ct\ \rs\ models is Borel.  We shall even isolate its
precise complexity.

To see that the classification problem for \ct\ \rs\ models is Borel,
we need only the most basic property of \rs\ models. Recall that the
{\emph standard system} of a nonstandard model $M\models \PA$ is the
collection
\[\SSy(M)\defeq\{X\cap\N: X\in\Def(M)\}\;.
\]
The following result is standard, see for instance Proposition~1.8.1
of \cite{ksbook} for a proof.

\begin{prop}
  \label{basic}
  If $M$ and $N$ are \rs\ models of a completion $T$ of \PA, then
  $M\cong N$ iff $\SSy(M)=\SSy(N)$.
\end{prop}

When $M$ is countable, $\SSy(M)$ is a countable set of reals, and
hence $\SSy(M)$ is coded by a real. We must now be more precise about
how we code countable sets of reals. Unfortunately, the space
$[\mathcal P(\omega)]^\omega$ of countable sets of reals does not
carry a natural standard Borel structure.  We work instead with the
space $\mathcal P(\omega)^\omega$ of countable \emph{sequences} of
reals, and let $E_\sset$ denote the equivalence relation defined on
$\mathcal P(\omega)^\omega$ by
\[x\mathrel{E}_\sset y \Iff \set{x(n):n\in\omega}=\set{y(n):n\in\omega}\;.
\]
(The relation $E_\sset$ has also assumed the names $\mathord{=}^+$,
$E_{\sf ctble}$ and $F_2$.)  It is easy to see that
$\mathrel{E}_\sset$ is a Borel equivalence relation.  Moreover,
Proposition~\ref{basic} shows that the map which sends a recursively
saturated model $M$ to a code for $\SSy(M)$ is a Borel reduction from
$\oiso_T^\srec$ to $E_\sset$.  This implies in particular that
$\oiso_T^\srec$ is Borel, and hence it is not nearly as complex as the
full $\oiso_T$.

\begin{thm}
  \label{recsat}
  The isomorphism equivalence relation $\oiso_T^\srec$ on the space of
  \rs\ models of $T$ is Borel bireducible with $E_\sset$.\footnote{The
    referee has pointed out that the nontrivial direction of Theorem
    \ref{recsat} is essentially the same as the main result of
    Marker's \cite{mar07}. Marker proved that for any first order
    theory in a \ct\ language where the type space $S(T)$ is un\ct,
    $E_\sset\leq \oiso_T$.}
\end{thm}

\begin{proof}
  We have just seen that there is a Borel reduction from
  $\oiso_T^\srec$ to $E_\sset$.  For the reverse direction, we shall
  need the notion of genericity.  If $(\N,\ldots)$ is any expansion of
  the standard model of arithmetic, then a subset $X\subseteq\N$ is
  said to be \emph{Cohen generic} over $(\N,\ldots)$ iff it meets
  every dense subset of the poset $2^{<\N}$ which is definable over
  $(\N,\ldots)$.  Cohen generics exist over every \ct\ expansion of
  $\N$.  We will work over $(\N,T)$, where we have identified $T$ with
  the set of G\"odel numbers of the sentences in $T$.

  Now, by Lemma~6.3.6 of \cite{ksbook}, there exists a perfect set
  $\S$ of subsets of $\N$ which are \emph{mutually Cohen generic} over
  $(\N,T)$ in the sense that for any distinct $X_1,\ldots,X_n\in\S$,
  $X_n$ is Cohen generic over $(\N,T,X_1,\ldots,X_{n-1})$. Identifying
  $\mathcal P(\omega)$ with the perfect set $\S$, each $C\in\mathcal
  P(\omega)^\omega$ naturally corresponds to an element
  $\S_C\in\S^\omega$.  Let $\X_C$ be the collection of subsets of $\N$
  which are definable from $T$ together with the sets enumerated in
  $\S_C$.  By mutual genericity, if $C\not= C'$, then
  $\X_C\not=\X_{C'}$. Since $\X_C$ is a Scott set and $T\in\X_C$,
  there exists a \ct\ \rs\ model $M_C$ of $T$ such that
  $\SSy(M_C)=\X_C$ (see for instance Theorem~3.5 of \cite{smo81}). it
  follows that the map $C\mapsto M_C$ is a Borel reduction from
  $E_\sset$ to $\oiso_T^\srec$, which completes the proof.
\end{proof}

The equivalence relation $E_\sset$ is an important benchmark in the
Borel reducibility hierarchy; many natural equivalence relations lie
at this complexity level. $E_\sset$ is not essentially countable, but
rather lies ``just above'' the countable Borel equivalence relations
(indeed, $E_\infty<_BE_\sset$ but there are few known interesting $E$
such that $E_\infty<_BE<_BE_\sset$).  In particular,
Theorem~\ref{recsat} implies that the class of \rs\ models is also not
essentially countable. There is, however, a simple argument of Jim
Schmerl which already implies this fact, and moreover implies that
many related classes of models are not essentially countable.

Let $T$ be a completion of \PA\ and let $M$ be a countable model of
$T$. If $A\subseteq\om$ is not in $\SSy(M)$, then by compactness, $M$
has an elementary extension $N$ such that $A\in \SSy(N)$.  In
particular, $N$ realizes a type which is not realized in M.  Moreover,
if $M$ is \rs, then we can make $N$ \rs\ as well. The following result
shows that a class of models with this property cannot be essentially
countable.

\begin{thm}
  \label{not_ec}
  Suppose that $\cal C$ is a class of \ct\ models such that every
  $K\in\cal C$ has an elementary extension in $\cal C$ realizing a
  type which is not realized in $K$. Further suppose that $\cal C$ is
  closed under unions of \ct\ elementary chains. Then $\cal C$ is not
  essentially \ct.
\end{thm}

\begin{proof}
  We shall use the characterization of essential countability provided
  by Theorem~\ref{hk}.  Let $F$ be a \ct\ fragment of
  $\lan_{\om_1,\om}$ and let $M$ be a model which is a union of a
  continuous elementary chain in $\cal C$, and which realizes
  uncountably many types. By a Skolem-L\"owenheim argument, for every
  finite (or even \ct) $\bar a\in M^n$, we have $M=\bigcup_{\al<\om_1}
  K_\al$, where $K_\al\in \cal C$ and $(K_\al,\bar
  a)\prec_F(K_\beta,\bar a)$ for all $\al<\beta<\om_1$. Hence, there
  must be $\al$ and $\beta$, such that $(K_\al,\bar
  a)\prec_F(K_\beta,\bar a)$ and $(K_\al,\bar a)\not\cong(K_\beta,\bar
  a)$.
\end{proof}

The paragraph preceding Theorem~\ref{not_ec} also applies to \ct\ \rs\
models of Presburger Arithmetic, which is the theory $\Th(\N,+)$. In
fact, it applies to the class of \ct\ \rs\ models of any
rich\footnote{$T$ is said to be \emph{rich} iff there exists a
  computable sequence of formulas $\seq{\vp_n(x):n\in\om}$ such that
  for all disjoint finite $A,B\subset \N$, $T\vdash\ex x\
  [\bigwedge_{i\in A} \vp_i(x)\land\bigwedge_{j\in B}\lnot\vp_j(x)]$.}
theory.  Hence, Schmerl's argument shows that none of these classes is
essentially \ct.

\section{Pairs of \rs\ models}

We have seen that the classification problem for countable recursively
saturated models is Borel.  However, each such model displays a rich
second-order structure which itself is a subject of further
classification attempts.  Much work has been done towards classifying
elementary submodels, elementary cuts, and automorphisms of \rs\
models of \PA.  None of these attempts have been completed, and there
are many open problems.  In this section we shall treat elementary
cuts, and in the next section automorphisms.

If $K$ is an \el\ cut in a \ct\ \rs\ model $M$ and $K$ itself is \rs,
then $K$ and $M$ will have the same standard system and hence $K\cong
M$.  Still, there are $\cont$ many isomorphism types of structures of
the form $(M,K)$, where $M$ and $K$ are \rs\ and $K\ee M$.  We shall
establish the following result.

\begin{thm}
  \label{cuts}
  Let $M$ be a \rs\ model of \PA. Then the classification problem for
  pairs $(M,K)$, where $K\ee M$ is \rs, is Borel complete.
\end{thm}

For the proof, we shall initially give a \emph{single} model $M$
satisfying the conclusion of Theorem~\ref{cuts}.  Afterwards, we will
indicate how to modify the construction to obtain the full result.

Let $S_\N$ be the set $\{\seq{\cgod{\vp},n}: \N\models\vp(n)\}$. If
$(M,S)$ is is an elementary extension of $(\N,S_\N)$, then $S$ is an
example of a nonstandard \emph{full inductive satisfaction class} for
$M$, \emph{i.e.},~$(M,S)\models \SPA$ and $S$ satisfies Tarski's
inductive definition of satisfaction for all formulas in the sense of
$M$.  The existence of a full inductive satisfaction class for a model
$M$ entails strong restrictions on $\Th(M)$, but $M$ does not have to
be an elementary extension of $\N$ (see \cite{kot91}).  The next two
lemmas, which we state just for elementary extensions of $(\N,S_\N)$,
have more general formulations with almost identical proofs.

\begin{lem}
  \label{satisfaction}
  If $(\N,S_\N)\prec (M,S)$ and the extension is proper, then $M$ is
  \rs.
\end{lem}

\begin{proof}[Sketch of proof]
  First let us notice that for each $\vp(v,\bar x)$, we have
  \[(\N,S_\N)\models\fo v\;\fo\bar x\;\left[\vp(v,\bar x)
    \iff\seq{\cgod{\vp},(v,\bar x)}\in S_\N\right]\;.
  \]
  It follows that the same holds in $(M,S)$. Let $p(v,\bar x)$ be a
  recursive type. Let $P(x)$ be a formula which defines the set of
  G\"odel numbers for the formulas in $p(v,\bar x)$. Suppose that for
  some $\bar b\in M$, $p(v,\bar b)$ is consistent. Then for each
  $n<\om$,
  \[(M,S)\models \ex v\;\fo \cgod{\vp}<n\;
  \left[P(\cgod{\vp})\arr \seq{\cgod{\vp},(v,\bar b)}\in S\right]\;.
  \]
  By overspill, this must be true in $M$ for all $n<c$, for some
  nonstandard $c$, and this shows that $p(v,\bar b)$ is realized in M.
\end{proof}

\begin{lem}
  \label{uniq}
  Suppose that $(M,S_0)$ and $(M,S_1)$ are each elementary extensions
  of $(\N,S_\N)$.  If $(M,S_0,S_1)\models\SPA$ (recall that this means
  $M$ satisfies the induction schema even for formulas that mention
  $S_0,S_1$), then $S_0=S_1$.
\end{lem}

\begin{proof}[Sketch of proof]
  Tarski's inductive definition of satisfaction is first-order over
  $(\N,S_\N)$. By elementarity, $S_0$ and $S_1$ obey the same
  definition in $M$.

  Now, by induction on complexity of formulas, one can show that for
  all formulas $\vp$ (in the sense of $M$) and all $\bar a\in M^n$,
  $\seq{\vp,\bar a}\in S_0\iff \seq{\vp, \bar a}\in S_1$. (Here, we
  used the assumption that $(M,S_0,S_1)\models\SPA$; in fact it is
  enough to assume that $(M,S_0,S_1)$ satisfies the
  $\Delta_0$-induction schema.)
\end{proof}

Now, let $(M,S_0)$ be a fixed \ct\ conservative elementary extension
of $(\N,S_\N)$. Then $M$ is recursively saturated, and since
$\SSy(M)=\Def(\N,S_\N)$, there is only one such $M$ up to
isomorphism. We shall show that this $M$ satisfies the conclusion of
Theorem~\ref{cuts}.

For a \ct\ linear order $(I,<)$, let $(\N(I+1),S)$ be the canonical
$(I+1)$-model of $\Th(\N,S_\N)$ with respect to some fixed minimal
type. This model is generated by an ordered set of indiscernibles
$\{a_i: i\in I+1\}$. Let $\N(I)$ be the elementary submodel generated
inside $(\N(I+1),S)$ by the set $\{a_i: i\in I\}$ (if $I$ is empty,
then put $\N(I)=\N$). Now, $\N(I+1)$ and $M$ are isomorphic as models
of \PA\ (without the satisfaction class), so we may let $f\colon
\N(I+1)\rightarrow M$ be a back-and-forth isomorphism and $K_I\defeq
f(M(I))$. This $K_I$ is the `canonical' $I$-cut of $M$. It is easy to
verify that the map $I\mapsto (M,K_I)$ is Borel.

We must show that this construction yields a Borel reduction from
linear orders to pairs of models.  To see that the isomorphism type of
$(M,K_I)$ depends only on the isomorphism type of $(I,<)$, first
observe that by the basic properties of canonical $I$-models, we have
$(\N(I+1),S)$ is a conservative \ele\ of $(\N(I),\N(I)\cap S)$. Thus,
it follows from Lemma~\ref{uniq} that $\N(I)\cap S$ is the only full
inductive satisfaction class of $\N(I)$ which is coded\footnote{If
  $K\subseteq M\models\PA$, then we say that a set $A\subseteq K$ is
  \emph{coded} in $M$, if $A=B\cap K$, for some $B\in \Def(M)$.} in
$\N(I+1)$.  Moreover $(\N(I),\N(I)\cap S)$ is an isomorphic copy of
the canonical $I$-extension of $(\N,S_\N)$. It follows that $K_I$ has
a unique full inductive satisfaction class $S_I$ which is coded in
$M$, and with the property that $(K_I,S_I)$ is an isomorphic copy of
the canonical $I$-extension of $(\N,S_\N)$.

To conclude the proof in this case, we must show that if $(J,<)$ is
another linear order and $g\colon(M,K_I)\rightarrow(M,K_J)$ is an
isomorphism, then $(I,<)\iso(J,<)$.  Again using Lemma~\ref{uniq}, we
have that $g(S_I)=S_J$, and hence that $(K_I,S_I)\cong(K_J,S_J)$.
Now, since the results discussed in Section~2 regarding canonical
$I$-models also hold for models of $\SPA$, we can conclude that
$(I,<)\cong (J,<)$.


In order to establish Theorem~\ref{cuts} for arbitrary $M$, we shall
require an additional fact.  A set $S\subseteq M$ is \emph{partial
  inductive satisfaction class} for a model $M\models \PA$ iff
$\seq{\cgod{\vp},a}$ is in $S$ iff $M\models \vp(a)$, for all formulas
$\vp(x)$ and all $a\in M$, and $(M,S)\models\SPA$.

\begin{thm}[Theorem~10.5.2 of \cite{ksbook}]
  \label{prime}
  Every \ct\ \rs\ model $N\models\PA$ has a partial inductive
  satisfaction class $S$ such that $(N,S)$ is the prime model of
  $\Th(N,S)$.
\end{thm}

To obtain the full version of Theorem~\ref{cuts}, we now modify the
above proof as follows. Instead of using $(\N,S_\N)$ and its canonical
$I$-extensions, we fix a \ct\ \rs\ $M\models\PA$ and select a prime
partial inductive satisfaction class $S$ for $M$ given by
Theorem~\ref{prime}.  There are $\cont$ many such classes, but this is
not a problem since in the construction $S$ will serve just as an
additional parameter. For a linear order $(I,<)$, we now take
$(M',S')$ to the $I+1$-canonical model of $\Th(M,S)$, and as before,
we take $K(I)$ to be the corresponding cut in $M$ (via an isomorphism
$f\colon M'\rightarrow M$). The rest of the argument is now similar,
but one has to be more careful. In Lemma~\ref{uniq}, $S_0$ and $S_1$
are \emph{full} inductive satisfaction classes, \emph{i.e.}, they
decide the ``truth'' of all formulas in the sense of the model, hence
the conclusion $S_0=S_1$ is easy to get. In the present setting we
cannot assume that $S$ is full.  The task can still be accomplished
with the aid of the more subtle Lemma~10.5.3 of \cite{ksbook} and its
corollary, which says that every \ct\ \rs\ model $M\models\PA$ has a
\ct\ \rs\ \ele\ $N$ such that for every end extension $N'$ of $N$ and
every embedding $f\colon N'\rightarrow N'$ such that $f(M)$ is cofinal
in $M$, $f\harpo M$ is the identity function.

\section{Conjugacy classes}

The automorphism groups of countable saturated structures have been
the subject of much study, and in many cases the conjugacy problem is
known to be Borel complete.  For example, the conjugacy problem for
the automorphism group of the rational linear ordering $(\Q,<)$, the
random graph, and the atomless Boolean algebra are all known to be
Borel complete (for a discussion of these results, see \cite{ces}). It
is shown in \cite{kkk91} that if $M$ is a \ct\ \rs\ model of \PA, then
\[\aut(\Q,<)\;\leq\;\aut(M)\;\leq\;\aut(\Q,\mathord{<})
\]
but $\aut(M)\not\cong\aut(\Q,\mathord{<})$. The group $\aut(M)$ is
known to have continuum many conjugacy classes, but little is known
about their classification.  What is known can be summarized as
follows.  For every $f\in\aut(M)$, let us set
\[\fix(f)\defeq\set{x\in M: f(x)=x},
\textup{ and }\;I_{\rm fix}(f)\defeq\set{x\in M:\fo y\leq x\ f(y)=y}\;.
\]
By a theorem of Smory\'nski \cite{smo82a}, a cut $I$ of a \ct\ \rs\
model of \PA\ is of the form $I_{\rm fix}(f)$ for some $f\in\aut(M)$
if and only if it is closed under exponentiation. Since each
nonstandard model has continuum many pairwise nonisomorphic (or even
not elementarily equivalent) cuts which are closed under
exponentiation, this immediately yields continuum many conjugacy
classes in \rs\ models.

If $M$ is \as\footnote{A \rs\ model $M\models \PA$ is said to be
  \emph{\as}\ iff $\SSy(M)$ is closed under \ar\
  definability. Arithmetic saturation is stronger than recursive
  saturation. Every \ct\ \as\ model has a cofinal extension which is
  \as\ and every \ct\ \as\ model has a cofinal extension which is \rs\
  but not \as.}, then this can be refined further by considering fixed
point sets of the automorphisms.  It is easy to see that $\fix(f)$ is
an elementary submodel of $M$.  Every \ct\ \rs\ model of \PA\ has
continuum many pairwise nonisomorphic elementary submodels, and by a
theorem of Enayat \cite{ena07}, if $M$ is \as\ then for every $K\prec
M$ there is an $f\in\aut(M)$ such that $\fix(f)\cong K$.  However, if
$M$ is not \as, then as shown in \cite{kkk91}, for every $f\in\aut(M)$
we have that $\fix(f)\cong M$.

It is known that for a \ct\ \rs\ model $M$, a cut $I\ee M$ is of the
form $\fix(f)$ for some $f\in\aut(M)$ if and only if $I$ is
\emph{strong} in $M$: for each function $f$ which is coded in $M$ and
such that $I\subseteq\dom(f)$, there is $c>I$ such that for all $i\in
I$, $f(i)>I$ iff $f(i)>c$.  However, we do not know in general which
elementary pairs $(M,K)$ are of the form $(M,\fix(f))$.  We now
establish the following consequence of Theorem \ref{cuts}.


\begin{thm}\label{conjugacy}
  For every \ct\ \rs\ model $M\models\PA$ the conjugacy equivalence
  relation on $\aut(M)$ is Borel complete.
\end{thm}

Of course, the conjugacy equivalence relation on $\aut(M)$ can be
identified with the isomorphism equivalence relation on the class of
pairs $(M,f)$ where $f$ is an automorphism of $M$.  Hence, it makes
sense to ask whether this relation is Borel complete.

\begin{proof}
  Let $(I,<)$ be a countable linearly ordered set. We will construct
  an ``$I$-canonical'' \au\ $f_I\in\aut(M)$. Let $I^+=(I,<)+(\Z,<)$,
  and let $(M',S')$ be the canonical $I^+$ model of $\Th(M,S)$, where
  $S$ is a partial inductive satisfaction class for $M$ given by
  Theorem~\ref{prime}. Let $\{a_i:i\in I^+\}$ be the generators of
  $(M',S')$. Let $f'$ be the automorphism of $M'$ generated by
  $a_i\mapsto a_i$, for $i\in I$, and $a_i\mapsto a_{i+1}$, for $i\in
  \Z$. Finally, let $f_I$ be the image of $f'$ under a back-and-forth
  isomorphism $g\colon M'\rightarrow M$. Then $\fix(f_I)=K(I)$ (where
  $K(I)$ is the `canonical' I-cut of $M$ defined in the previous
  section).  If $(I,<)$ and $(J,<)$ are \ct\ linearly ordered sets,
  and $f_I$ and $f_J$ are conjugate then $(M, \fix(f_I))\cong
  (M,\fix(f_J))$. By Theorem~\ref{cuts} we must have $(I,<)\cong
  (J,<)$, and the result follows.
\end{proof}

\bibliographystyle{alpha}
\begin{singlespace}
  \bibliography{reductions}
\end{singlespace}

\end{document}